\documentclass[11pt]{article}

\usepackage{amsmath}
\usepackage{amssymb}

 \newcommand{\beq}{\begin{equation}}
\newcommand{\eeq}{\end{equation}}

\numberwithin{equation}{section}

\newtheorem{theorem+}           {Theorem}   
\newtheorem{definition+}    {Definition}
\newtheorem{lemma+}    {Lemma}
\newtheorem{corollary+}    {Corollary}
\newtheorem{proposition+}  {Proposition}
\newtheorem{example+}    {Example}

\newenvironment{theorem}{\begin{theorem+}\sl}{\end{theorem+}\rm}
\newenvironment{definition}{\begin{definition+}\rm}{\end{definition+}\rm}
\newenvironment{lemma}{\begin{lemma+}\sl}{\end{lemma+}\rm}
\newenvironment{corollary}{\begin{corollary+}\sl}{\end{corollary+}\rm}
\newenvironment{proposition}{\begin{proposition+}\sl}{\end{proposition+}\rm}
\newenvironment{example}{\begin{example+}\rm}{\end{example+}\rm}
\newenvironment{proof}{\medbreak\noindent{\it Proof.}\rm}{\hfill$\square$\rm}

\newcommand{\Psvx}{{ \Psi_{v,x}}}
\newcommand{\Pspx}{{ \Psi_{\vph,x}}}

\newcommand{\Psuo}{{ \Psi_{u,0}}}
\newcommand{\nux}{\nu_u(x)}
\newcommand{\nuxa}{\nu_u(x,a)}
\newcommand{\tnuph}{\tilde\nu(u,\vph)}
\newcommand{\nuph}{\nu(u,\vph)}
\newcommand{\suph}{\sigma(u,\vph)}
\newcommand{\mrph}{\mu_r^\vph}

\newcommand{\obl}{\Omega}

\newcommand{\okn}{1\le k\le n}
\newcommand{\vph}{\varphi}

\renewcommand{\Bbb}{\mathbb}
\newcommand{\Z}{{\Bbb  Z}}

\newcommand{\R}{{ \Bbb R}}

\newcommand{\Rnm}{{ \Bbb R}_-^n}
\newcommand{\Rnp}{{ \Bbb R}_+^n}
\newcommand{\C}{{\Bbb  C}}
\newcommand{\Cn}{{\Bbb  C\sp n}}

\newcommand{\D}{{\Bbb D}}

\newcommand{\N}{{\cal N}}

\newcommand{\cO}{{\cal O}}
\newcommand{\cH}{{\cal H}}
\newcommand{\cJ}{{\cal J}}
\newcommand{\cI}{{\cal I}}

\newcommand{\m}{{\mathfrak m}}

\begin{document}

\baselineskip=17pt

\begin{center}
{\Large\bf An extremal problem for generalized Lelong numbers}
\end{center}

\begin{center}
{\large Alexander Rashkovskii}
\end{center}

\begin{abstract} We look for pointwise bounds on a
plurisubharmonic function near its singularity point, given the
value of its generalized Lelong number with respect to a
plurisubharmonic weight. To this end, an extremal problem is considered. In certain cases, the problem is solved explicitly.



\end{abstract}

\section{Introduction}\label{sec:intro}

Let $\m=\{f\in\cO_0:\:f(0)=0\}$ be the maximal ideal in the local ring $\cO_0$ of germs of analytic functions at $0\in\Cn$. It is not hard to see that the mixed multiplicity $e_{n-1}(\cJ,\m)$ (due to Teissier and Risler \cite{T} and Bivi\`{a}-Ausina \cite{BA}) of an ideal $\cJ$ of $\cO_0$ and $n-1$ copies of $\m$ is at least $p$, $p\in\Z_+$, if and only if $\cJ\subset\m^p$; this follows from the fact that $e_{n-1}(\cJ,\m)$ equals the minimal multiplicity of functions $f\in\cJ$.

Let now $\cI$ be an $\m$-primary ideal of $\cO_0$ (which means that the common zero set of functions from $\cI$ is $\{0\}$); what can be said about an ideal $\cJ$ if $e_{n-1}(\cJ,\cI)\ge p$?

If the ideal $\cI$ is generated by functions $g_1,\ldots,g_l$ ($l\ge n$) and the ideal $\cJ$ is generated by functions $f_1,\ldots,f_m$, then the value $e_{n-1}(\cJ,\cI)$ equals the residual mass of the Monge-Amp\`ere current $dd^c\log|f|\wedge(dd^c\log|g|)^{n-1}$ at $0$, which is, actually, the generalized Lelong number (due to Demailly) $\nu(\log|f|,\log|g|)$ of the plurisubharmonic function $\log|f|$ with respect to the plurisubharmonic weight $\log|g|$; here $|f|^2=\sum|f_j|^2$ and similarly for $g$ and all other vector functions (mappings) below.

Now we can restate our question in the category of plurisubharmonic functions as one on asymptotic behavior of the functions with given values of their generalized Lelong numbers: {\sl Given a plurisubharmonic weight $\vph$, what can be said about (the asymptotic of) a plurisubharmonic function $u$
near $\vph^{-1}(-\infty)$ if $\nuph\ge c$?}
\medskip

When $\vph(z)=\log|z-x|$, $x\in\Cn$, the value $\nu(u,\vph)$ is the classical Lelong number $\nu_u(x)$ of $u$ at $x$. In this case, we have the bound $u(z)\le \nu_u\log|z-x|+O(1)$
as $z\to x$, and $\nu_u(x)\ge c$ if and only if $u\le c\,\log|z-x|+O(1)$ near $x$, which is a plurisubharmonic analogue to the remark above on $e_{n-1}(\cJ,\m)$.

More generally, let $a\in\Rnp$ and
\beq
\phi_{a,x}(z)=\max_k \,a_k^{-1}\log|z_k-x_k|. \label{eq:phax} \eeq
Then $u(z)\le\nuxa\phi_{a,x}(z) + O(1)$ with $\nuxa$ the directional Lelong
numbers due to Kiselman at $x$ for  $a\in\Rnp$.
In terms of Demailly's generalized Lelong numbers $\nuph$ with
respect to plurisubharmonic weights $\vph$, this gives $$u\le
\tau^{-1} \nu(u,\phi_{a,x})\phi_{a,x} +O(1),$$ where $\tau=(a_1\dots
a_n)^{-1}= (dd^c\phi_{a,x})^n(\{x\})$ is the Monge-Amp\`ere mass of
$\phi_{a,x}$ at $x$.

In the general case, {\it the relation} \beq\label{eq:conj} u(z)\le
\tau_\vph^{-1} \nu(u,\vph)\vph(z) +O(1),\quad z\to x,\eeq
$\tau_\vph$ been the residual mass of $(dd^c\vph)^n$ at
$x=\vph^{-1}(-\infty)$, {\it need not be true} for all $u$, even for the weights $\vph$
that are maximal outside $x$.

\begin{example} Let $\vph=\log|f|$ with $f(z_1,z_2)=(z_1^3-z_2^3,z_1z_2)$;
we have $(dd^c\vph)^2=6\delta_0$. The function $u(z)=\log|z_1|$
satisfies $\nuph=3$, so the inequality (\ref{eq:conj}) would take
the form $\log|z_1|\le \frac12\vph +O(1)$, which is not true for
$z=(z_1,0)$, $z_1\to 0$.
\end{example}

To answer our question for an arbitrary maximal weight $\vph$, we consider an extremal problem whose solution
$d_\vph$ gives the best upper bound for functions $u$ satisfying $\nu(u,\vph)\ge c$:
\beq\label{eq:mr}u(z)\le
c\,\tau_\vph^{-1}d_\vph(z) +O(1).\eeq
The extremal function $d_\vph$ is
plurisubharmonic and maximal outside the singularity point of
$\vph$, so it is a fundamental solution for the complex Monge-Amp\`ere operator $(dd^c)^n$.
It turns out that, unless $\vph$ is such that (\ref{eq:conj}) is true for all $u$, one cannot {\it characterize} the condition $\nu(u,\vph)\ge c$ by means of any upper bound on $u$. That is why relation (\ref{eq:mr}) need not imply $\nu(u,\vph)\ge c$; what it does imply is that $u$ is the upper envelope of the family of functions with this property.

Then we turn to the question of explicit construction of this extremal function.
Let us assume that a weight $\vph$ has asymptotically analytic behavior, that is, for any $\epsilon >0$, $$(1+\epsilon)c_\epsilon\log|g_\epsilon|\le\vph\le (1-\epsilon)c_\epsilon\log|g_\epsilon|\quad{\rm near\ }x$$ with $g_\epsilon$ a holomorphic mapping, $g_\epsilon(x)=0$, and $c_\epsilon>0$. The class of such weights is quite large; actually, we have no example of a maximal weight that does not have asymptotically analytic singularity.
 Given such a weight $\vph$, the function $d_\vph$ can be constructed by means of plurisubharmonic functions $\phi_i$ generating so called {\it Rees valuations}, i.e., generic multiplicities of pullbacks of analytic functions on exceptional primes $E_i$ of log resolutions. More precisely, $d_\vph$ is presented as the (regularized) upper envelope of the family of largest plurisubharmonic minorants for the functions $\min_i\gamma_i\phi_i$, where  $\gamma_i>0$ are such that $\sum m_i\gamma_i=1$ for certain $m_i$ determined by the weight $\vph$. As a consequence, it is continuous (as a map to $\R\cup\{-\infty\}$) and has asymptotically analytic singularity.

When $\vph$ has homogeneous (in logarithmic
coordinates) singularity, the asymptotic of $d_\vph$ can be
determined explicitly; this reflects, in particular, the fact that the largest plurisubharmonic minorant for the minimum of homogeneous plurisubharmonic functions can be computed easily. The extremal function is found by means of an integral representation for the generalized Lelong numbers with homogeneous weights \cite{R4}.  The asymptotic turns out to be simplicial, equivalent to $\phi_{a,x}$ with $a\in\Rnp$ such that $a_k=\tau_\vph^{-1} \nu(\log|z_k-x_k|,\vph)$, $1\le k\le n$.

A weight $\vph=\log|g|$ with analytic singularity given by a mapping $g:\C_0^n\to\C_0^n$ is equivalent to a homogeneous weight if and only if the multiplicity of $g$ at $0$ equals $n!$ times the covolume of the Newton polyhedron of $g$ at $0$ \cite{R7}, and the latter holds, by Kouchnirenko's theorem, for generic mappings $g$ with given Newton polyhedron. Furthermore, if all the components of $g$ are monomials, then $\vph=\log|g|$ is homogeneous. Thus, the above result on homogeneous singularities solves explicitly the problem on ideals $\cJ$ satisfying the condition $e_{n-1}(\cJ,\cI)\ge p$ for a {\it monomial} ideal $\cI$.

In addition, the homogeneous case indicates that
the functions $d_\vph$ form rather a small subclass of Green-like
plurisubharmonic functions, containing, of course, the class of weights $\vph$ for which (\ref{eq:conj})
holds true for every plurisubharmonic function $u$ (and thus $d_\vph=\vph+O(1)$); we call such
functions {\it flat}. In particular, a homogeneous weight is shown to be flat if
and only if it has simplicial asymptotic. So, for homogeneous weights, the functions $d_\vph$ are precisely flat weight. It is plausible to conjecture that this holds in general situation as well, however we were not able to prove it.

We characterize flat weights $\vph$ by the relation
$$\nu(\max_j\,u_j,\vph)=\min_j\,\nu(u_j,\vph)$$ for all
plurisubharmonic functions $u_j$. The proof is based on the following result that has independent interest: If a plurisubharmonic function $u$ dominates maximal weight $\vph$ and $\nu(u,\vph)=\nu(\vph,\vph)$, then $u=\vph+O(1)$.
A sufficient "inner" condition for a
weight to be flat is also given. In addition, flat weights are used
to get upper bounds for $d_\vph$ in the general case.

\medskip
The paper is organized as follows. The next section contains basic facts
on generalized Lelong numbers and Green functions. In Section~\ref{sec:extr} we construct the
extremal function $d_\vph$ as the upper envelope of negative plurisubharmonic
functions $u$ with given value of $\nuph$. This function is computed
in Section~\ref{sec:ahw} for the case of homogeneous singularity, and in Section~\ref{sec:tame} it is described for weights with asymptotically analytic singularities.
In Section~\ref{sec:flat} we study the class of flat weights.
Finally, in the last section we address a few open questions.

\section{Preliminaries}\label{sec:pre}

Given a domain $\obl\subset\Cn$, let $PSH(\obl)$ denote the class of
all plurisubharmonic functions on $\obl$ and $PSH^-(\obl)$ be its
subclass consisting of the nonpositive functions. If $x\in\Cn$,
$PSH_x$ will mean the collection of all germs of plurisubharmonic
functions near $x$.

We recall that a function $u\in PSH(\obl)$ is called {\it maximal}
on $\obl$ if for any $v\in PSH(\obl)$ the relation $\{v>u\}\Subset
\obl$ implies $v\leq u$ on $\obl$. A locally bounded $u$ is maximal on $\obl$ if and only if
$(dd^cu)^n\equiv 0$ there.

\subsection{Lelong numbers and generalizations}\label{ssec:LN}

The Lelong number of $u\in PSH_x$ at the point $x\in\Cn$ is
$$
\nux=\lim_{r\to 0}\int_{|z-x|<r} dd^c u\wedge (dd^c\log|z-x|)^{n-1};
$$
here $d=\partial + \bar\partial,\ d^c= ( \partial
-\bar\partial)/2\pi i$.  It can also be calculated as
\beq
\label{eq:form0} \nux=\liminf_{z\to x}\frac{u(z)}{\log|z-x|}
\eeq
and
\beq
\label{eq:form1} \nux=\lim_{r\to -\infty}r^{-1}\lambda_u(x,r), \eeq
where $\lambda_u(x,r)$ is the mean value of $u$ over the sphere
$|z-x|=e^r$, see \cite{Kis1}.
\medskip

A more detailed information on the behavior of $u$ near $x$ can be
obtained by means of the {\it directional Lelong numbers} due to
Kiselman \cite{Kis2}
$$
\nuxa = \liminf_{z\to x}\frac{u(z)}{\phi_{a,x}(z)}=
\lim_{r\to-\infty}r^{-1}\lambda_u(x,ra), $$
where $a=(a_1,\ldots,a_n)\in\Rnp$, the function $\phi_{a,x}$ is defined by (\ref{eq:phax}), and
$\lambda_u(x,ra)$ is the mean value
of $u$ over the set $\{z:\: |z_k-x_k|=e^{ra_k},\ \okn\}$. In
particular, $\nux=\nu_u(x, (1,\ldots,1))$.
\medskip

A general notion of the Lelong number with respect to a
plurisubharmonic weight was introduced and studied by J.-P.~Demailly
\cite{D1}, \cite{D}. Let $\vph$ be a continuous plurisubharmonic
function near $x\in\Cn$, locally bounded outside $x$, and
$\vph^{-1}(-\infty)=\{x\}$; we can assume $\vph\in PSH(\Cn)\cap
L_{loc}^\infty(\Cn\setminus\{x\})$. Such functions are called
{\it weights} (centered at $x$), and the collection of the weights will be denoted by $W_x$. A weight $\vph\in W_x$ is be called {\it maximal} if it is a maximal plurisubharmonic function on a punctured
neighborhood of $x$ (i.e., satisfies $(dd^c\vph)^n=0$ outside $x$). We denote the class of all
maximal weights centered at $x$ by $MW_x$; when we want to specify that $\vph$ is maximal on
$\omega\setminus\{x\}$, we write $\vph\in MW_x(\omega)$.

Denote $B_r^\vph=\{z:\:\vph(z)<r\}$. The value \beq \label{eq:gen}
\nuph=\lim_{r\to -\infty}\int_{B_r^\vph} dd^c u\wedge
(dd^c\vph)^{n-1}= dd^c u\wedge (dd^c\vph)^{n-1}(\{x\})\eeq is called
{\it the generalized Lelong number}, or {\it the Lelong--Demailly
number}, of $u$ with respect to the weight $\vph$.

We list here some basic tools for dealing with the generalized Lelong numbers with respect to weights $\vph\in W_x$, see details in \cite{D}.

\begin{theorem}\label{theo:Dsc} If
$u_k\to u$ in $L^1_{loc}$, then
$\limsup_{k\to\infty}\nu(u_k,\vph)\le \nu(u,\vph)$.
\end{theorem}

\begin{theorem}\label{theo:CT2} If
$\limsup\,\frac{v(z)}{u(z)}=l<\infty$ as $z\to x$, then $\nu(v,\vph)\le
l\,\nu(u,\vph)$.
\end{theorem}

\begin{theorem}\label{theo:CT1} If
$\limsup\,\frac{\vph_1(z)}{\vph_2(z)}=l<\infty$ as $z\to x$, then
$\nu(u,\vph_1)\le l^{n-1}\,\nu(u,\vph_2)$.
\end{theorem}

Denote $\vph_r=\max\{\vph,r\}$; the measure
$\mrph=(dd^c\vph_r)^n-\chi_r(dd^c\vph)^n $ on the pseudosphere
$S_r^\vph=\{z:\: \vph(z)=r\}$ is called the {\it swept out
Monge-Amp\`ere measure} for $\vph$; here $\chi_r$ is the
characteristic function of the set $\Cn\setminus{B_r^\vph}$. Any
$u\in PSH(B_R^\vph)$ is $\mrph$-integrable for $r<R$, and satisfies
the following relation, which we will call the
Lelong--Jensen--Demailly formula, \beq\label{eq:LJDf} \mrph(u)-
\int_{B_r^\vph}u (dd^c\vph)^n=\int_{-\infty}^r
\int_{B_t^\vph}dd^cu\wedge (dd^c\vph)^{n-1}. \eeq

If $ \vph\in MW_x(B_R^\vph\setminus \{x\})$, then the
function $r\mapsto \mrph(u)$ is convex on $(-\infty,R)$ and
  \beq\label{eq:nuph} \nuph=\lim_{r\to
-\infty}r^{-1}\mrph(u), \eeq which is an analogue to formula
(\ref{eq:form1}). In particular, $\nux=\nu(u,\log|\cdot-x|)$ and
\beq \label{eq:nuxa} \nuxa=a_1\ldots a_n\,\nu(u,\phi_{a,x}) \eeq
with the weights $\phi_{a,x}$ defined by (\ref{eq:phax}).
\medskip

Yet another generalization of the notion of Lelong number, developing its presentation (\ref{eq:form0}), are {\it relative types} introduced in \cite{R7}.
For any function $u\in PSH_x$, we denote its {\it type
relative to a weight} $\vph\in MW_x$ as
\beq\label{eq:reltype} \sigma(u,\vph)=\liminf_{z\to
x}\frac{u(z)}{\vph(z)}.\eeq  Maximality of $\vph$
implies the bound
\beq\label{eq:rtbound}u\le\sigma(u,\vph)\vph+O(1).\eeq
The types have properties similar to those given in Theorems \ref{theo:Dsc}--\ref{theo:CT2}, and are related to the Lelong-Demailly numbers by the inequality
\beq\label{eq:sunu}\nu(u,\vph)\ge\tau_\vph\,\sigma(u,\vph),\eeq
where $\tau_\vph=\nu(\vph,\vph)$, which follows from Theorem \ref{theo:CT2}.

\subsection{Almost homogeneous weights}

Let a nonpositive plurisubharmonic function $\Phi$ in the unit
polydisk $\D^n$ satisfy the relation \beq \label{eq:hom}
\Phi(z)=\Phi(|z_1|, \ldots,|z_n|)= c^{-1}\Phi(|z_1|^c,
\ldots,|z_n|^c)\quad\forall c>0. \eeq  It is a continuous function
in $\D^n$, and $(dd^c\Phi)^n=0$ on $\{\Phi>-\infty\}$. Such
functions arise as plurisubharmonic characteristics for local behavior of
plurisubharmonic functions near their singularity points. Namely,
given a plurisubharmonic function $v$, its (local) {\it indicator}
at a point $x$ is a plurisubharmonic function $\Psvx$ in $\D^n$ such
that for any $y\in\D^n$ with $y_1\cdot\ldots\cdot y_n\neq 0$, \beq
\Psvx(y)=-\nu_v(x,a),\quad a=-(\log|y_1|,\ldots,\log|y_n|)\in\Rnp.
\label{eq:lind} \eeq This function satisfies (\ref{eq:hom}) and
is the largest nonpositive plurisubharmonic function in $\D^n$ whose
directional Lelong numbers at $0$ coincide with those of $v$ at $x$,
so \beq v(z)\le\Psvx(z-x)+O(1) \label{eq:bound} \eeq near $x$, see
the details in \cite{LeR}, \cite{R1}.

A function $\vph\in PSH(\obl)$ with $\vph^{-1}(-\infty)=x\in\obl$ is
said to be {\it almost homogeneous} at $x$ if it is asymptotically
equivalent to its indicator $\Psi_{\vph,x}$ \cite{R4}, that is,
\beq\label{eq:ahw} \exists\lim_{z\to
x}\frac{\vph(z)}{\Psi_{\vph,x}(z-x)}=1.\eeq

A weight $\vph=\log|g|$ generated by a holomorphic mapping $g$ with isolated zero at $x$ was proved in \cite{R7} to be almost homogeneous if and only if the multiplicity of $g$ at $x$ equals the Monge-Amp\`ere mass of $(dd^c\Psi_{\log|g|,x})^n$ of the indicator of $\log|g|$.

Theorem~\ref{theo:CT1}
reduces computation of the generalized Lelong--Demailly numbers with
respect to almost homogeneous weights $\vph$ to those with respect
to the homogeneous weights $\Psi_{\vph,x}$; moreover, as shown in
\cite{R4}, $\nuph=\nu(\Psi_{u,x},\Psi_{\vph,0})$.

The swept out Monge-Amp\`ere measure for homogeneous weights can be
determined by the following procedure, see \cite{R4}. A function
$\Phi\in PSH^-(\D^n)$ satisfying (\ref{eq:hom}) generates the
function $f(t):=\Phi(e^{t_1},\ldots,e^{t_n})$,
convex and positive homogeneous in $\Rnm$. Given a subset
$F$ of the convex set \beq\label{eq:setL}L^\Phi=\{t\in\Rnm: f(t)\le
-1\},\eeq we put
\beq \Theta_F^\Phi=\{\lambda b:\: 0\le\lambda\le 1,\ b\in \Rnp,\
\sup_{t\in F}\langle b,t\rangle = \sup_{t\in L^\Phi}\langle
b,t\rangle=-1\}. \label{eq:gth} \eeq A measure $\gamma^\Phi$ on
$L^\Phi$ is defined as \beq\label{eq:gamma}\gamma^\Phi(F)= n!\, {\rm
Vol}\,\Theta_F^\Phi,\quad F\subset L^\Phi.\eeq

\begin{theorem}\label{theo:int}{\rm \cite{R4}}
Let $\vph$ be an almost homogeneous weight centered at $x$. Then for
any $u\in PSH_x$,
\beq\label{eq:repmc}
\nuph=\int_{E^\Phi}\nu_u(x,-t)\,d\gamma^\Phi(t)
\eeq
where the measure $\gamma^\Phi$ on the set $E^\Phi$ of extreme
points of the set $L^\Phi$ (\ref{eq:setL}) is defined by
(\ref{eq:gamma}) and (\ref{eq:gth}) with $\Phi=\Pspx$.
\end{theorem}

\subsection{Asymptotically analytic weights}\label{ssec2.3}

Let $\phi\in W_x$ have analytic singularity, i.e.,
$\phi=c\log|F|+O(1)$ near $x$, where $c>0$ and $F$ is a
holomorphic mapping of a neighbourhood of $x$ to $\C^N$ with
isolated zero at $x$. As is known, the integral closure of the ideal
generated by the components $F_j$ of $F$ has precisely $n$
generators $\xi_k=\sum a_{k,j}F_j$ (generic linear combinations of
$F_j$), $k=1,\ldots,n$, so $\phi=c\log|\xi|+O(1)$. By
Theorem~\ref{theo:CT1}, we can then assume $\phi=c\log|\xi|$ and consider $F$ to be
equidimensional and so, $\phi\in MW_x$, which we will tacitly do in the sequel. The collection of all weights with analytic singularities at $x$ will be denoted by $AW_x$.

As follows from Demailly's approximation theorem \cite{D2}, any $\phi\in W_x$ can be approximated by weights $\phi_k\in AW_x$ in a neighborhood $D$ of $x$ such that
\beq\label{eq:Dap}
\phi(z)-\frac{C}{k}\le \phi_k(z)\le \sup_{|\zeta-z|<r}\phi(\zeta)+\frac1k\log\frac{C}{r^n},\quad z\in U.
\eeq
Specifically,
$$\phi_k=\frac1k\sup\{\log|f|:\:f\in\cO(D),\ \int_D |f|^2e^{-2k\phi}\,dV<1\}=\frac2k\log\sum_i|f_{k,i}|^2,$$
where $\{f_{k,i}\}_i$ is an orthonormal basis for the Hilbert space
\beq\label{eq:Hkp}\cH(k\phi)=\{\log|f|:\:f\in\cO(D),\ |f|e^{-k\phi}\in L^2(D)\}.\eeq

A weight  $\psi\in W_x$ will be called {\it asymptotically analytic} ($\psi\in AAW_x$) if for every $\epsilon>0$ there exists a weight $\phi_\epsilon\in AW_x$ such that
\beq\label{eq:ass}(1+\epsilon)\phi_\epsilon +O(1)\le  \psi\le (1-\epsilon)\phi_\epsilon +O(1);\eeq
by Theorem~\ref{theo:CT1}, we get then
\beq\label{eq:app}(1-\epsilon)^{n-1}\nu(u,\phi_\epsilon) \le\nu(u,\psi) \le (1+\epsilon)^{n-1}\nu(u,\phi_\epsilon).\eeq

According to \cite{BFaJ}, a weight $\vph\in W_x$ is called {\it
tame} if there exists a constant $C>0$ such that for every $t>C$ and
every analytic germ $f$ from the multiplier ideal ${\cal J}(t\vph)$
of $t\vph$ at $x$ (that is, the function $fe^{-t\vph}$ is
$L^2$-integrable near $x$), one has $\log|f|\le (t-C)\vph+O(1)$. For
maximal weights $\vph$, the latter can be written as
$\sigma(\log|f|,\vph)\ge t-C$, where $\sigma$ is the relative type
(\ref{eq:reltype}).
Let $\vph$ be tame, and let $\vph_k$ be Demailly's approximations of
$\vph$. By \cite[Lemma 5.6]{BFaJ}, they satisfy  \beq\label{eq:Dapt}  \vph+O(1)\le\vph_k\le (1-C_\vph/k)\vph+O(1)\eeq
near $x$ and therefore $\vph$ has asymptotically analytic singularity. Moreover, conditions (\ref{eq:Dapt}) characterize tame weights.

One of the main results of \cite{BFaJ} is the following integral
representation for the Lelong--Demailly numbers with respect to tame
weights, which is an extension of the representation (\ref{eq:repmc}): \beq\label{eq:reptame} \nuph=-\int_{\cal
V}g_u\,MA(g_\vph),\eeq where $g_u$ and $g_\vph$ are certain {\it formal plurisubharmonic functions} on the space $\cal V$ of all centered
normalized valuations and $MA(g_\vph)$ is a positive measure on
$\cal V$. We refer to  \cite{BFaJ} for precise definitions.

We will also need the following simple result on tame weights.

\begin{proposition}\label{prop:tmin} If $\phi$ and $\psi$ are tame maximal weights and $\alpha,\beta\in\R_+$, then the greatest plurisubharmonic minorant $\vph$ of the function $h=\min\,\{{\alpha\phi}, \beta\psi\}$ is a tame maximal weight and the constant $C_\vph$ in (\ref{eq:Dapt}) is independent of $\alpha$ and $\beta$.\end{proposition}

\begin{proof} Maximality of $\vph$ outside $x$ is evident.
Let $\vph_k$, $\phi_k$, $\psi_k$ denote Demailly's approximating weights for $\vph$, $\phi$, and $\psi$, respectively.
 Then $$\vph+O(1)\le \vph_k\le\alpha\phi_k\le (1-C_\phi/k)\alpha\phi+O(1)$$ and similarly with $\beta\psi$,  so $\vph_k\le (1-C_\vph/k)h+O(1)$ with $C_\vph=\max\{C_\phi,C_\psi\}$, which implies
 $$\vph+O(1)\le \vph_k\le (1-C_\vph/k)\vph+O(1)$$
 and thus the tameness of $\vph$.
\end{proof}

\subsection{Pluricomplex Green functions}
We will use the following extremal function introduced (for the case of continuous singularity) by V.
Zahariuta \cite{Za0} (see also \cite{Za}); for the general case, see \cite{R7}. Let $\obl\subset\Cn$ be a
bounded hyperconvex domain. Given a plurisubharmonic function $\vph$, locally bounded and maximal  outside $x\in\obl$, let \beq\label{eq:gzf}
G_\vph(z)=G_{\vph,\obl}(z)=\sup\{u(z):\: u\in PSH^-(\obl),\ u\leq \vph+O(1)\
{\rm near\ }x\}.\eeq This function is plurisubharmonic in $\obl$,
maximal in $\obl\setminus x$, $G_{\vph,\obl}=\vph+O(1)$ near $x$,
and $G_{\vph,\obl}(z)\to 0$ as $z\to\partial\obl$; moreover, it is a
unique plurisubharmonic functions with these properties.
Furthermore, if $\vph$ is continuous (and so, $\vph\in MW_x$), $G_{\vph,\obl}$ is continuous on $\obl\setminus x$. We
will refer to this function as the {\it Green} (or {\it Green--Zahariuta}) {\it function
with singularity} $\vph$.

If $\vph(z)=\log|z-x|$, then  $G_{\vph,\obl}$ is just the standard
pluricomplex Green function $G_{x,\Omega}$ of $\obl$ with pole at
$x$.

When $\vph$ is the indicator of a plurisubharmonic function $v$, the
corresponding Green--Zahariuta function can be alternatively
described as the upper envelope of all nonpositive plurisubharmonic
functions $u$ in $\obl$ such that $\nu_u(x,a)\ge\nu_v(x,a)$ for all
$a\in\Rnp$ \cite{LeR}.

Since any analytic weight is equivalent to a maximal analytic weight (see Section~\ref{ssec2.3}), the Green functions $G_{\phi_k}$ for Demailly's approximations $\phi_k$ of a weight $\phi\in MW_x$ are well defined, too. If $\phi$ is a tame weight, then (\ref{eq:Dapt}) implies
\beq\label{eq:Dapt1} G_\phi\le G_{\phi_k}\le (1-{C_\phi}/{k})G_\phi.\eeq


\section{Extremal functions for Lelong--Demailly numbers}\label{sec:extr}

Given a function $u\in PSH_x$, it is convenient to consider its {\it
normalized} Lelong--Demailly numbers with respect to weights
$\vph\in W_x$,
$$ \tnuph=\tau_\vph^{-1}\nuph,$$
where $$\tau_\vph:=\nu(\vph,\vph)=(dd^c\vph)^n(\{x\})>0$$ is the residual Monge-Amp\`ere mass of $\vph$.
We have, in particular, $\tilde\nu(\vph,\vph)=1$ and
$\tilde\nu(cu,c\vph)=\tnuph$ for all $c>0$.

We will be concerned with upper bounds of functions $u$ in terms of
$\tnuph$. To this end, it looks reasonable to fix a bounded
hyperconvex neighbourhood $\obl$ of $x$ and consider the upper
envelope of the class $$\N^0_{\vph,\Omega}=\{u\in PSH^-(\obl):\:\tnuph\ge 1\}.$$
Note however that it need not be closed under
the operation $(u,v)\mapsto\max\,\{u,v\}$. Indeed, as follows from
Theorem~\ref{theo:CT2}, $$
\nu(\max\,\{u,v\},\vph)\le\min\,\{\nuph,\nu(v,\vph)\},$$ and the
inequality can be strict.

\begin{example}\label{ex:1} Take
the weight $\vph=\max\{3\log|z_1|, 3\log|z_2|, \log|z_1z_2|\}$ in
$\D^2$; we have then $\tau_\vph=6$. The functions
$u_j(z)=2\log|z_j|$ satisfy $\tilde\nu(u_j,\vph)=1$, while
$\tilde\nu(\max_ju_j,\vph)=2\tau_\vph^{-1}\nu_\vph(0)=2/3<1$.
\end{example}

Furthermore, we would like to work also with plurisubharmonic
functions whose definition domains are proper subsets of $\obl$.
That is why we introduce the class
$$\N_{\vph,\Omega}=\bigcup_{k=1}^\infty \N^k_{\vph,\Omega},$$
where for $k\ge 1$,
$$\N^k_{\vph,\Omega}=\{u\in PSH^-(\Omega):\:
u\le\max_{1\le j\le k}\,u_j\ {\rm in\ }\omega\ni x,\ u_j\in \N^0_{\vph,\omega},\ \omega\Subset\obl\}.$$
Observe that the set $\N_{\vph,\Omega}$ is convex, because if $u_i=\max_j u_{ij}$, $i=1,2$, then
$$\alpha u_1+(1-\alpha)u_2=\max_{j,l}\{ \alpha u_{1j}+(1-\alpha)u_{2l}\}.$$

\begin{definition} Given a weight $\vph\in W_x$, $x\in\obl$, let
\beq \label{eq:dvph}d_{\vph,\Omega}(z)=\limsup_{y\to z}\sup\{u(y):\:
u\in\N_{\vph,\Omega}\},\quad z\in\obl.
\eeq
\end{definition}

Note that the function $d_{\vph,\Omega}$ need not belong to
$\N_{\vph,\Omega}$; this is the main difference with the
construction of "usual" pluricomplex Green functions.

\begin{proposition}\label{theo:hull}
Let $\obl$ be a bounded hyperconvex domain in $\Cn$, and let
$\vph\in W_x$, $x\in\obl$. Then
\begin{enumerate} \item[(i)] $d_{\vph,\Omega}\in PSH^-(\obl)\cap
L_{loc}^\infty(\obl\setminus\{x\})$ and is maximal in
$\obl\setminus\{x\}$; \item[(ii)] $d_{\vph,\Omega}(z)\to 0$ as
$z\to\partial\obl$;
\item[(iii)]
$\nu(d_{\vph,\Omega},\phi)=\inf\{\nu(u,\phi):\: u\in
\N_{\vph,\Omega}\}$ for any weight $\phi\in W_x$;
\item[(iv)] $u\le \tnuph\,
d_{\vph,\Omega}$ in $\obl$ for all $u\in PSH^-(\obl)$, and $u\le
\tnuph\, d_{\vph,\Omega}+O(1)$ near $x$ for every $u\in PSH_x$.
\end{enumerate}
\end{proposition}

\begin{proof}
By the Choquet lemma, there is a sequence $u_j\in \N_{\vph,\Omega}$
increasing to a function $h$ such that $h^*=d_{\vph,\Omega}$. Since
$\vph-\sup_\Omega\vph\in\N_{\vph,\Omega}$, we can assume $u_j\in
L_{loc}^\infty(\obl\setminus\{x\})$ for all $j$. Given a ball
$B\Subset\obl\setminus\{x\}$, consider the functions
$$
v_j(z)=\sup\{u(z)\in PSH^-(\obl):\: u\le u_j \ {\rm in}\
\obl\setminus B\}.
$$
Then $v_j\in\N_{\vph,\Omega}$ and satisfy $(dd^cv_j)^n=0$ in $B$.
Since $v_j\ge u_j$, the functions $v_j$ increase a.e. to
$d_{\vph,\Omega}$ and so, $(dd^cd_{\vph,\Omega})^n=0$ in $B$. This
proves (i).

Assertion (ii) holds because the Green function $G_\vph$ belongs to $\N_{\vph,\Omega}$.

By Theorem \ref{theo:CT2},
$\nu(d_{\vph,\Omega},\phi)\le\inf\{\nu(u,\phi): u\in
\N_{\vph,\Omega}\}$ for every weight $\phi\in W_x$. On the other hand, the above functions
$u_j\in\N_{\vph,\Omega}$ converge to $d_{\vph,\Omega}$ in
$L_{loc}^1$ (the monotone convergence theorem), and
Theorem~\ref{theo:Dsc} gives
$\nu(d_{\vph,\Omega},\phi)\ge\limsup\,\nu(u_j,\phi)$, which proves
(iii).

The first relation in (iv) is obvious, and the second one follows
from the fact that for any $u\in PSH_x$ with $\tnuph\ge 1$, the
function $\max\,\{u,\vph\}$ can be extended from a neighbourhood of
$x$ to $\obl$ as a bounded above plurisubharmonic function.
 \end{proof}

\begin{corollary} If $\nu(u,\vph)\ge c>0$, then
$u(z)\le c\,\tau_\vph^{-1}d_{\vph,\Omega}(z) +O(1)$ as $z\to x$ for any bounded hyperconvex neighbourhood $\Omega$ of $x=\vph^{-1}(-\infty)$.
\end{corollary}

For almost homogeneous weights $\vph$, asymptotics of the extremal
functions $d_{\vph,\Omega}$ can be computed explicitly and they turn
out to be simplicial (see the next section). In the general case, it
is then likely that for the functions $d_{\vph,\Omega}$ form rather
a small subclass of Green-like functions as well; it would be nice
to get a description of such functions.

The function $d_{\vph,\Omega}$ can be estimated by means of {\it
flat} weights defined as follows. Let $\sigma(u,\phi)$ denote the type of $u$ relative to $\phi$, see Section~\ref{ssec:LN}.

\begin{definition}
A maximal weight $\phi\in MW_x$ is {\it flat} if
$\tilde\nu(u,\phi)=\sigma(u,\phi)$ for any function $u\in PSH_x$.
\end{definition}

As follows from (\ref{eq:nuxa}), the directional weights
$\phi_{a,x}$ are flat. More properties of flat weights will be given
in Section~\ref{sec:flat}.

Evidently, $d_{\phi,\Omega}= G_{\phi,\Omega}$ (the Green--Zahariuta function) for any flat weight
$\phi$. This gives the following simple bound.

\begin{proposition}\label{prop:lowerhull} If $\vph\in W_x$,
$x\in\Omega$, then $d_{\vph,\Omega}\le \tau_\phi^{-1}\tau_\vph
G_{\phi,\Omega}$ and, consequently,
$\tilde\nu(d_{\vph,\Omega},\vph)\ge \tau_\phi^{-1}\tau_\vph$ for every
flat weight $\phi$ satisfying $\phi\le\vph+O(1)$ near $x$.
\end{proposition}

\begin{proof} Let $\tnuph\ge 1$. By
Theorem~\ref{theo:CT1}, we have $\nu(u,\phi) \ge\nuph\ge \tau_\vph$
for any flat weight $\phi$ satisfying $\phi\le\vph+O(1)$. Therefore
$\sigma(u,G_{\phi,\Omega})=\sigma(u,\phi)=\tilde\nu(u,\phi)\ge
\tau_\phi^{-1}\tau_\vph$, which implies the statements in view of
(\ref{eq:rtbound}) and Theorem~\ref{theo:CT2}.
\end{proof}

\begin{corollary}\label{cor:lowerhull} If $\vph\in W_x$ is such that
$\vph\ge N\log|z-x|$ near $x$ for some $N>0$, then
$d_{\vph,\Omega}\le N^{1-n}\tau_\vph G_{x,\Omega}$, where
$G_{x,\Omega}$ is the pluricomplex Green function of $\obl$ with
logarithmic pole at $x$.
\end{corollary}

{\it Remarks.} 1. If $\vph\in AW_x$,
$\vph=\log|F|+O(1)$ for a holomorphic mapping $F$ with isolated zero
of multiplicity $m$ at $x$, then Corollary~\ref{cor:lowerhull} gives
the bound $d_{\vph,\Omega}\le L^{1-n}m\, G_{x,\Omega}$, where $L>0$
is the \L ojasiewicz exponent of $F$ at $x$, i.e., the infimum of
$\gamma>0$ such that $|F(z)|\ge |z-x|^\gamma$ near $x$.

2. By analogy with the analytic case, we will call the value
\beq\label{eq:lyap}L_\vph=\limsup_{z\to x}
\frac{\vph(z)}{\log|z-x|}\eeq the \L ojasiewicz exponent of the
weight $\vph$. Corollary~\ref{cor:lowerhull} implies
$d_{\vph,\Omega}\not\equiv 0$ for $\vph$ with finite $L_\vph$. It is
easy to construct weights with infinite \L ojasiewicz exponent,
however we have no examples of {\sl maximal} weights $\vph$ with
$L_\vph=\infty$. Moreover, we do not know if there exists weights
$\vph\in W_x$ such that $d_\vph\equiv 0$.

3. Another upper bound for weights with finite \L ojasiewicz
exponent will be given in Corollary~\ref{cor:Loind}.

\section{The case of almost homogeneous weights}\label{sec:ahw}

Here we will show that if $\vph\in W_x$ is an almost homogeneous
weight (\ref{eq:ahw}), then the function $d_{\vph,\Omega}$ (\ref{eq:dvph}) is flat. Moreover, it has a
simplicial asymptotic $d_{\vph,\Omega}=\phi_{a,x}+O(1)$ with $a\in
\Rnp$ determined explicitly.

Let us assume $x=0$. Given a weight $\vph\in W_0$, we set
\beq\label{eq:sk} a_k=\tilde\nu(\log|z_k|,\vph),\quad \okn.\eeq

\begin{theorem}\label{theo:ahw} Let $\Omega$ be a bounded
hyperconvex domain in $\Cn$, $0\in\Omega$, and let $\vph$ be an
almost homogeneous plurisubharmonic weight centered at $0$.
Then $d_{\vph,\Omega}$ equals the Green--Zahariuta function for the
singularity $\phi_{a,0}$ (\ref{eq:phax}) in $\Omega$ with $a\in\Rnp$
defined by (\ref{eq:sk}).
\end{theorem}

\begin{proof}
Take any $u\in PSH_0$ with $\tnuph\ge 1$; by Theorem~\ref{theo:int},
$$ \tnuph=\int_{E^\Phi}\nu_u(0,-t)\,d\tilde\gamma^\Phi(t),$$
where $\tilde\gamma^\Phi=\tau_\vph^{-1}\gamma^\Phi$. By the
definition of the indicator, $\nu_u(0,-t)=-\psi(t)$, where
$\psi(t)=\Psuo(e^{t_1},\ldots, e^{t_n})$. Therefore, the condition
$\tnuph\ge 1$ implies
$$ \int_{E^\Phi}\psi(t)\,d\tilde\gamma^\Phi\le -1.$$

The function $\psi$ is the restriction to $\Rnm$ of the supporting
function  of a convex subset $S$ of $\Rnp$: $\psi(t)=\sup\,\{\langle
b,t\rangle:\:b\in S\}$. This gives us
$$-1\ge \int_{E^\Phi}\sup_{b\in S}\langle b,t\rangle
\,d\tilde\gamma^\Phi  \ge  \sup_{b\in S}\langle b,\int_{E^\Phi}
t\,d\tilde\gamma^\Phi\rangle=\sup_{b\in S}\langle b,\mu\rangle,$$
where the vector $\mu\in\Rnm$ has the components
$$\mu_k=\int_{E^\Phi} t_k\,d\tilde\gamma^\Phi, \quad k=1,\ldots,n.$$
Therefore the set $S$ lies in the half-space $\Pi_\mu=\{b:\langle
b,\mu\rangle \le -1\}$ and so, $\psi(t)$ is dominated by the
supporting function $\psi_\mu$ of the set $\Pi_\mu\cap\Rnp$. It is
easy to see that $\psi_\mu(t)=\max_k t_k/|\mu_k|$.

As follows from Theorem~\ref{theo:int} applied to the functions
$\log|z_k|$, $\mu_k = -a_k$ with $a_k$ defined by (\ref{eq:sk}), and
thus $\psi(t)\le \max_k t_k/a_k$. This means that
$\Psuo\le\phi_{a,0}$.

Let now $v\in\N_{\vph,\obl}$ have the form $v=\max_ju_j$ near $0$
and $\tilde\nu(u_j,\vph)\ge 1$. Then
$\Psvx=\max_j\Psi_{u_j,x}\le\phi_{a,0}$. Therefore, in view of
(\ref{eq:bound}), $v$ is dominated by the Green--Zahariuta function
$G$ for the singularity $\phi_{a,0}$, so $d_{\vph,\obl}\le G$.

Finally, the relations $\tilde\nu(a_k^{-1}\log|z_k|,\vph)=1$ imply
$d_{\vph,\obl}\ge \phi_{a,0}+C$ with some constant $C$. Since
$G=\phi_{a,0}+O(1)$ near $0$, the equality $d_{\vph,\obl}= G$
follows from the uniqueness property for the Green--Zahariuta
functions, which completes the proof.
\end{proof}

\medskip
This can be used for estimation of the functions $d_\vph$ for
arbitrary weights $\vph\in W_0$ with finite \L ojasiewicz exponent
(\ref{eq:lyap}). It is easy to see that such weights are characterized by
the property \beq\label{eq:Loind} l:=\limsup_{z\to
0}\frac{\vph(z)}{\Psi_{\vph,0}(z)}<\infty,\eeq just because the
indicator of every weight has finite \L ojasiewicz exponent.

\begin{corollary}\label{cor:Loind} If $\vph\in W_0$ satisfies
(\ref{eq:Loind}), then $d_{\vph,\Omega}\le l^{1-n}G_a$, where $G_a$ is
the Green--Zahariuta function for the singularity $\phi_{a,0}$
in $\Omega$ with $a\in\Rnp$ defined by
(\ref{eq:sk}).
\end{corollary}

\begin{proof} Denote $\Psi=\Psi_{\vph,0}$ and take any $u\in PSH_x$
with $\tnuph\ge 1$. By Theorem~\ref{theo:CT1}, condition
(\ref{eq:Loind}) implies the inequality $\nu(u,\Psi)\ge
l^{1-n}\tau_\vph$. Therefore, by Theorem~\ref{theo:ahw}, $u\le
l^{1-n}\tau_\vph\tau_\Psi^{-1}\phi_{b,0}+C$ near $0$, where
$b_k=\tilde\nu(\log|z_k|,\Psi)$, $k=1,\ldots,n$. This implies the
statement because $\tau_\vph\tau_\Psi^{-1}\phi_{b,0}\le\phi_{a,0}$
with $a$ defined by (\ref{eq:sk}).
\end{proof}

\medskip
In an analytic setting, Theorem \ref{theo:ahw} gives us the following result on monomial ideals. Given an ideal $\cJ$ and an $\m$-primary ideal $\cI$ of $\cO_0$, denote by $e(\cI)$ the Samuel multiplicity of $\cI$ and by $e_{n-1}(\cJ,\cI)$ the mixed mixed multiplicity, due to Teissier and Risler \cite{T} and Bivi\`{a}-Ausina \cite{BA}, of $\cJ$ and $n-1$ copies of $\cI$. Denote, furthermore, $e_k(\cI)=e_{n-1}(\m_k,\cI)$ for the principal ideal $\m_k$ generated by the function $z_k$.

Let $\cI$ be an $\m$-primary monomial ideal generated by monomials $g_1,\ldots,g_l$, $l\ge n$, and let $\vph=\max_j\log|g_j|$; observe that $\vph$ is an indicator. Then $e(\cI)=\tau_\vph$ (because in the monomial case the both values equal $n!$ times the covolume of the Newton polyhedron of the mapping $g$; in the general case, the equality is proved, for example, in \cite{D8}). Furthermore, if an ideal $cJ$ is generated by functions $f_1,\ldots,f_m$, then $e_{n-1}(\cJ,\cI)=\nu(\log|f|,\vph)$, where $f=(f_1,\ldots,f_m)$ (this follows from multilinearity of the mixed multiplicities).

The numbers $a_k$ from (\ref{eq:sk}) can now be computed now as
$$a_k=e_{n-1}(\m_k,\cI)\,e^{-1}(\cI),$$ where $\m_k$ is the principal ideal generated by the function $z_k$,
and Theorem \ref{theo:ahw} takes the following form.

\begin{theorem} Let $\cI$ be an $\m$-primary monomial ideal in $\cO_0$. If an ideal $\cJ$ satisfies
$e_{n-1}(\cJ,\cI)\ge p$, then $\cJ$ is contained in the ideal generated by the functions $z_k^{p_k}$, where $$p_k=\min\,\{q\in \Z_+:\: p/q \le e_{n-1}(\m_k,\cI)\},  \quad k=1,\ldots, n.$$
\end{theorem}


\section{The case of (asymptotically) analytic weights}\label{sec:tame}

Given $\phi=c\log|F|\in AW_x$, denote by ${\mathfrak a}$ the
primary ideal generated by $F_1,\ldots,F_n$. By \cite[Prop. 4.10]{BFaJ}, there exists a simple modification $\pi:X_\pi\to X$
above a neighbourhood $X$ of $x$ that is an isomorphism above
$X\setminus x$, with a normal crossing exceptional divisor
$\pi^{-1}(x)$, such that $\pi^{-1}{\mathfrak a}$ is a principal
ideal and the measure $MA(g_\phi)$ is a finite sum of weighted Dirac
measures with masses $c_i=c^{n-1}\,I_i$ at the divisorial valuations
({\it Rees valuations}) ${\rm ord}_{E_i}$  over the irreducible components $E_1,\ldots,E_N$
of $\pi^{-1}(x)$. Here ${\rm ord}_{E_i}(f)$ is the vanishing order
of  $f\circ\pi$ along $E_i$, and $I_i$ is the intersection number of
the $n$-tuple $(E_i,\pi^{-1}{\mathfrak a},\ldots,\pi^{-1}{\mathfrak
a})$.
Therefore, (\ref{eq:reptame}) takes the form
$$\tilde\nu(u,\phi)=\sum_{1\le i\le N} a_{i}\nu_{i}(u),\quad u\in PSH(\Omega),$$ where
$a_i=\tau_\phi^{-1}c_i$ and $\nu_{i}(u)$ is the generic Lelong number
of the function $u\circ\pi$ along the set $E_i$.

By \cite[Thm. 4.3]{R7}, there exist maximal weights $\phi_i$ such that $a_i\nu_{i}(u)=\sigma(u,\phi_i)$ for every $u$, and in \cite[Thm. 5.13]{BFaJ} these weights are proved to be continuous and tame. We will call them {\it elementary representing weights} for the weight $\phi$. Tameness of these weights implies $\sigma(\phi_k,\phi_j)>0$ for all $k,j$.

If $\tilde\nu(u,\phi)\ge 1$, then
$$\sum_{1\le i\le N} \sigma(u,\phi_i)\ge 1.$$
This implies existence of positive numbers $\beta_1,\ldots,\beta_N$ with $\sum_i\beta_i\ge1$, such that
$\sigma(u,\phi_i)\ge\beta_i$ and thus, by (\ref{eq:rtbound}), $u(z)\le \beta_i\,\phi_i(z)+O(1)$ near $x$. Therefore,
\beq\label{eq:beta}u(z)\le \phi_{\langle\beta\rangle}+O(1),\quad z\to x,\eeq
where $\phi_{\langle\beta\rangle}$ denotes the greatest plurisubharmonic minorant of the function $\min_i\,\beta_i\phi_i$. Notice that $\sigma(\phi_{\langle\beta\rangle},\phi_i)\ge\beta_i\ge \inf_{k,j}\sigma(\phi_k,\phi_j)>0$ for all $i$.

Furthermore, let $\beta_j^0$ denote the the least lower bound for $\sigma(u,\phi_j)$ over all $u$ with  $\tilde\nu(u,\phi)\ge 1$. Then
\beq\label{eq:beta0}\phi_{\langle\beta\rangle}\le \phi_{\langle\beta^0\rangle},\quad \beta\in \Pi_1,\eeq
where $\Pi_1=\{\beta\in\R^N:\sum_i\beta_i\ge 1\}$. Note that $\beta^0$ need not belong to $\Pi_1$.

By Proposition~\ref{prop:tmin}, $\phi_{\langle\beta\rangle}$ is a tame maximal weight and the constant $C_\vph$ in (\ref{eq:Dapt}) for $\vph=\phi_{\langle\beta\rangle}$ is independent of $\beta$; let us call it $C$. Then the Green--Zahariuta function $G_{\langle\beta\rangle}$ of $\Omega$ for the singularity
$\phi_{\langle\beta\rangle}$ satisfies
\beq\label{eq:tamebeta} G_{\langle\beta\rangle}\le G_{\langle\beta\rangle,k}\le (1-C/k)G_{\langle\beta\rangle},\eeq
where $G_{\langle\beta\rangle,k}$ are the Green functions for Demailly's approximations of $G_{\langle\beta\rangle}$.

Any function $u\in\N_{\phi,\obl}$ satisfies $u\le\max_j\,u_j$ for finitely many functions $u_j$ with $\tilde\nu(u_j,\phi)\ge 1$. Since $u$ is negative in $\obl$, (\ref{eq:beta}) applied to each $u_j$ implies the relation
$u\le \max\,\{G_{\langle\beta\rangle}:\: \beta\in B_u\}$,
where $B_u$ is a finite subset of the half-space $\Pi_1$, depending on $u$.
Therefore, by (\ref{eq:beta0}),
$$d_{\phi,\obl}(z)\le \sup\,\{G_{\langle\beta\rangle}(z):\: \beta\in \Pi_1\}=G_{\langle\beta^0\rangle}(z).$$
Since for $\beta\in\Pi_1$, $$\tilde\nu(G_{\langle\beta\rangle},\phi)=\tilde\nu(\phi_{\langle\beta\rangle},\phi)
=\sum_{1\le i\le N} \sigma(\phi_{\langle\beta\rangle},\phi_i)\ge 1$$ and thus $G_{\langle\beta\rangle}\in\N^0_{\phi,\obl}$, we have actually the equality $d_{\phi,\obl}= G_{\langle\beta^0\rangle}$.

In addition, (\ref{eq:tamebeta}) with $\beta=\beta^0$ implies
$$ d_{\phi,\obl}\le d_{\phi,\obl,k}+O(1)\le (1-C/k)d_{\phi,\obl}$$ and thus its continuity and tameness.

\medskip

Finally, if $\psi$ is an asymptotically analytic weight, then (\ref{eq:app}) implies
$$(1+\epsilon)^{n-1}d_{\vph_\epsilon,\obl}\le d_{\psi,\obl}\le (1-\epsilon)^{n-1}d_{\vph_\epsilon,\obl},$$
thus $d_{\psi,\obl}$ is asymptotically analytic (since so are $d_{\vph_\epsilon,\obl}$) and
$$ d_{\psi,\obl}=\sup_{\epsilon>0}\,(1+\epsilon)^{n-1}d_{\vph_\epsilon,\obl}$$
which gives, in particular, its continuity.

As a result, we have got the following properties of the extremal functions for analytic and asymptotically analytic weights.

\begin{theorem} If $\psi$ is an asymptotically analytic weight and $\obl$ is a bounded hyperconvex domain, then the function $d_{\psi,\obl}$ is continuous and has asymptotically analytic singularity. If $\psi$ has analytic singularity, then $d_{\psi,\obl}$ is tame.
\end{theorem}


\section{Flat weights}\label{sec:flat}
Here we consider the situation when the function
$d_{\vph,\Omega}$ has the same asymptotic as $\vph$. Assuming
$\vph$ to be maximal, it is easy to see that this happens if and
only if the weight $\vph$ is flat (i.e., satisfying $\tnuph=\suph$
for all $u$, see Section~\ref{sec:extr}). We denote the class of all
flat maximal weights centered at $x$ by $FW_x$.

\medskip
As follows from the definition of relative types, flat weights
$\vph$ satisfy the relation
$\nu(\max\,\{u,v\},\vph)=\min\,\{\nu(u,\vph),\nu(v,\vph)\}$ for any
plurisubharmonic functions $u$ and $v$. We will show that this is a
characteristic property of the class of flat weights (see
Corollary~\ref{cor:trop} below). Moreover, it suffices to check it
only on $u$ with $\tnuph\ge 1$ and $v=\vph$
(Corollary~\ref{cor:extr}).

The crucial step in proving the claim is the following extremal
property of maximal weights.

\begin{lemma}\label{lemma:maxpr} Let $\vph\in MW_x(B_R^\vph)$. If
$v\in PSH(B_R^\vph)$ is such that $v\ge\vph$ in $B_R^\vph$, $v(x)\to
R$ as $x\to\partial B_R^\vph$, and $\nu(v,\vph)=\tau_\vph$, then
$v\equiv\vph$.
\end{lemma}

\begin{proof} We can assume $R=0$, then $\vph=
G_{\vph,D}$ with $D=B_0^\vph$.

Given $\epsilon>0$, let $w_\epsilon=\max\{v-\epsilon,\vph\}$. Since
$w_\epsilon=\vph$ near $\partial D$, we have \beq\label{eq:prm1}
\int_{D} dd^cw_\epsilon\wedge (dd^c\vph)^{n-1}= \int_{D}
(dd^c\vph)^{n}=\tau_\vph.\eeq The inequality $v\ge\vph$ implies
$\vph\le w_\epsilon\le v$. Since
$\nu(w,\vph)=\nu(v,\vph)=\tau_\vph$, this gives the relation
$\nu(w_\epsilon,\vph)=\tau_\vph$, so
$$ \int_{D} dd^cw_\epsilon\wedge
(dd^c\vph)^{n-1}= \tau_\vph+ \int_{D\setminus\{x\}}
dd^cw_\epsilon\wedge (dd^c\vph)^{n-1}.$$ Comparing this with
(\ref{eq:prm1}), we get \beq\label{eq:prm2} dd^cw_\epsilon\wedge
(dd^c\vph)^{n-1}=\tau_\phi\delta_x.\eeq By the
Lelong--Jensen--Demailly formula (\ref{eq:LJDf}),
$$ \mu_r^\vph(w_\epsilon)=\mu_{r_0}^\vph(w_\epsilon)+
\int_{r_0}^r\int_{B_t^\vph} dd^cw_\epsilon\wedge
(dd^c\vph)^{n-1}\,dt,\quad r<r_0<0.$$ Choose $r_0=r_0(\epsilon)<0$
such that $w_\epsilon=\vph$ near $\partial B_{r_0}^\vph$, then
$\mu_{r_0}^\vph(w_\epsilon)=r_0\tau_\vph$, while the second term, by
(\ref{eq:prm2}), equals $(r-r_0)\tau_\vph$.  Therefore
$\mu_r^\vph(w_\epsilon)=r\tau_\vph$, $r<r_0$. By the construction of
the function $w_\epsilon$, this means
$$ \int \max\{v-\epsilon -r, 0\}\,d\mu_r^\vph=0,\quad r<r_0,$$
which implies \beq\label{eq:prm3} \mu_r^\vph(\{v-\epsilon >
r\})=0,\quad  \ r<r_0.\eeq By Demailly's maximum principle
\cite[ Thm.~5.1]{D3}, $\sup\,\{v(z)-\epsilon:\vph(z)<r\}$ equals the
essential supremum of $v-\epsilon$ with respect to the measure
$\mu_r^\vph$, so $v-\epsilon\le r$ in $B_r^\vph$ for all $r<r_0$ and
$\sigma(v,\vph)=\sigma(v-\epsilon,\vph)\ge 1$. Therefore, $v\le
G_{\vph,D}= \vph$.
\end{proof}

\begin{corollary}\label{cor:extr} If a weight $\vph\in MW_x$ is such that
$\tilde\nu(\max\,\{w,\vph\},\vph)=1$ for every $w\in PSH_x$ with
$\tilde\nu(w,\vph)\ge 1$, then $\vph\in FW_x$.
\end{corollary}

\begin{proof} Take any $w\in PSH_x$ with
$\tilde\nu(w,\vph)\ge 1$ and a real $R$ such that  $w<0$ in $B_R^\vph$ and $\vph\in MW_x(B_R^\vph)$. Then the function
$v:=\max\{w+R,\vph\}$ satisfies the conditions of
Lemma~\ref{lemma:maxpr}, so $v\equiv\vph$ and consequently,
$w\le\vph-R$ in $B_R^\vph$; in other words, $\sigma(w,\vph)\ge 1$.

Given now $u\in PSH_x$ with $\nu(u,\vph)>0$, the
function $w=[\tilde\nu(u,\vph)]^{-1} u$ satisfies
$\tilde\nu(w,\vph)=1$ and so, as we have just shown, $\sigma(w,\vph)\ge 1$, which means
$\tilde\nu(u,\vph)\le\sigma(u,\vph)$. In view of (\ref{eq:sunu})
this completes the proof.
\end{proof}

\begin{corollary}\label{cor:trop} If $\vph\in MW_x$ is such that
$\nu(\max_j\,u_j,\vph)=\min_j\,\nu(u_j,\vph)$ for any
plurisubharmonic functions $u_j$, $j=1,2$, then $\vph\in FW_x$.
\end{corollary}

We can also give some "inner" sufficient conditions for a weight to
be flat. Let us assume that there exist numbers $K\ge 0$, $\tau>0$,
and a family of plurisubharmonic functions
$\{\vph_x\}_{x\in\omega}$, where $\omega$ is a bounded domain in $\Cn$, such that
 \beq\label{eq:f1}\vph_x\in MW_x(\omega),\eeq
\beq\label{eq:f3}\vph_y(x)\le
K+\max\,\{\vph_z(x),\vph_z(y)\},\eeq
\beq\label{eq:f4}\tau_{\vph_x}=\tau\eeq
for all $x,y,z\in\omega$.

For example, this is the case for a family $\vph_x(y)=\vph(y-x)$ with a weight $\vph\in MW_0(\Omega)$ in a neighbourhood $\Omega$ of $0\in\Cn$ that
satisfies \beq\label{eq:sum} \vph(y-x)\le
K+\max\,\{\vph(x),\vph(y)\},\quad x,y\in\omega.\eeq

Since  $\vph_x(x)=-\infty$, (\ref{eq:f3}) implies, in particular,
\beq\label{eq:minus} \vph_y(x)\le K+ \vph_x(y);\eeq together with (\ref{eq:f3}), this
means that the function
$\rho(x,y)=\exp(\vph_y(x))$ is a non-Archimedean quasi-metric on $\omega$, that is,
$$\rho(x,x)=0,\quad \rho(x,y)\le
C\rho(y,x),\quad \rho(x,y)\le C\max\{\rho(x,z),\rho(y,z)\}.$$

\begin{theorem}\label{theo:equality} Let  a family of
weights $\vph_x$, $x\in\omega$, satisfy conditions (\ref{eq:f1})--(\ref{eq:f4}). Then $\vph_x\in FW_x$.
\end{theorem}

\begin{proof} Fix any $x\in\omega$; for brevity, we denote $\vph_x$ just by $\vph$. We can assume $\tau=\tau_\vph=1$; then,
in view of (\ref{eq:sunu}), it is only the relation
$\suph\ge\nuph$ to be proved.

Let $r<0$ be such that $B_r^\vph=\{y:\vph(y)<r\}\subset\omega$ and let
$t<r-3K$. Fix any $y\in\omega$ with $\vph(y)=t$. Then relation
(\ref{eq:f3})  implies \beq\label{eq:genupper} \vph_y(z)\le
K+\max\,\{\vph(z),t\},\quad z\in B_r^\vph.\eeq In particular,
for all $z$ with $t<\vph(z)\le r$ we get $ \vph_y(z)\le
K+\vph(z)$. On the other hand, for $z$ with $\vph(z)>2K+t$
relations (\ref{eq:f3}) and (\ref{eq:minus}) give
$$\vph(z)\le K+\max\,\{\vph_y(z),\vph_y(0)\}\le
K+\max\,\{\vph_y(z),K+t\},$$ which implies $\vph_y(z)\ge
\vph(z)-K$. Therefore $ \vph(z)-K\le\vph_y(z)\le
\vph(z)+K$ near the boundary of $B_r^\vph$; in particular,
\beq\label{eq:boundest} r-K\le\vph_y(z)\le r+K,\quad z\in
S_r^\vph.\eeq Furthermore, on a neighbourhood of $S_t^\vph$ we have,
in view of (\ref{eq:genupper}), \beq\label{eq:uppers} \vph_y(z)\le
2K+t.\eeq

Let $G$ denote the Green--Zahariuta function of $B_r^\vph$ with the
singularity $\vph_y$. Then $G(z)\le \vph_y(z)+O(1)$ near $y$ and
$$(dd^cG)^n=(dd^c\vph_y)^n=\delta_y.$$ Since the
function $\vph_y$ is maximal on $B_r^\vph\setminus \{y\}$, relation
(\ref{eq:boundest}) implies $ G\le \vph_y-r+K$ in $B_r^\vph$;
on $S_t^\vph$ we have then, by (\ref{eq:uppers}), $ G\le
3K+t-r$. Therefore \beq\label{eq:mbound} G(z)\le
M(\vph(z)-r),\quad z\in S_t^\vph,\eeq where
$M=M(r,t)=(3K+t-r)(t-r)^{-1}$.

We set
$$ \psi(z)=\left\{\begin{array}{ll}\vph(z)-r, &
\mbox{$z\in B_t^\vph$}\\
\max\,\{\vph(z)-r, M^{-1}G(z)\}, & \mbox{$z\in B_r^{\vph}\setminus
B_t^\vph$.}\end{array}\right. $$Due to (\ref{eq:mbound}), $\psi\in
PSH^-(B_r^\vph)$. Then it is dominated by the Green--Zahariuta
function for $B_r^\vph$ with the singularity $\vph$, that is, by
$\vph-r$, and so,
$$ M^{-1}G(z)\le
\vph(z)-r,\quad z\in B_r^{\vph}\setminus B_t^\vph.$$ By Demailly's
comparison theorem \cite[Thm.~3.8]{D5} for the swept out Monge-Amp\`ere measures, this implies the inequality $d\mu^{\vph-r}\le
M^{-n}d\mu^G$ and thus, as $d\mu^{\vph-r}=d\mu_r^\vph$,
\beq\label{eq:compar} d\mu_r^{\vph}\le M^{-n}d\mu^G.\eeq

The Lelong--Jensen--Demailly formula (\ref{eq:LJDf})
shows that for any function $u$ plurisubharmonic in a
neighbourhood of $\overline{B_r^\vph}$ and every $\tau<0$,
$$\mu^G(u)\ge \mu_\tau^G(u)\ge \int_{B_\tau^G}u(dd^c
G)^n=u(y).$$ Assuming $u$ negative in $B_r^\vph$,
inequality (\ref{eq:compar}) gives then
$$ \mu_r^\vph(u)\ge   M^{-n}
u(y)$$ and, since the only condition on $y$ is that $\vph(y)=t$, $$
\sup\,\{u(z):\: z\in B_t^\vph\}\le M^{n}\mu_r^\vph(u).$$

Now let $t=(1+\epsilon)r$, $\epsilon>0$, then $M=1+3K
\epsilon^{-1}r^{-1}$ satisfies $M\ge 1-\epsilon$ for all $r<
-3K\epsilon^{-2}$, so
$$ \sup\,\{u(z):\: z\in B_{(1+\epsilon)r}^\vph\}\le (1-\epsilon)^{n}
\mu_r^\vph(u)$$ and, by the definition of the relative type and
(\ref{eq:nuph}),
$$\suph\ge (1-\epsilon)^{n}(1+\epsilon)^{-1}
\nuph.$$ Letting $\epsilon\to 0$ completes the proof. \end{proof}

\medskip

\begin{example}\label{ex:11} It is easy to see that the weights $\vph=\phi_{a,0}$
(\ref{eq:phax}) satisfy (\ref{eq:sum}) with $K=(\min
a_k)^{-1}\log 2$ and therefore are flat. In particular, this
recovers Kiselman's result that the directional Lelong number can be
calculated by means of both the maximal and mean values.
\end{example}

\begin{example}\label{ex:12}  An almost homogeneous weight $\vph\in W_x$ is flat if and only if
its indicator $\Psi_{\vph,x}$ is simplicial, which follows from
Theorem~\ref{theo:ahw}. This does not hold true without the homogeneity assumption.
\end{example}

\begin{example}\label{ex:13}  Let $f$ be an irreducible holomorphic function on a neighbourhood
of $0\in\C^2$ such that $f(0)=0$ and $\{f=0\}$ is transverse to
$\{z_1=0\}$. By \cite[Prop.~3.6 and 3.9]{FaJ}, the weight
$\vph=\log\max\{|z_1|^s,|f|\}$ is flat for any $s\ge {\rm mult}_0f$.
We do not know if this can be deduced from Theorem~\ref{theo:equality}.
\end{example}

\begin{example}\label{ex:14}  If $F$ is a biholomorphic mapping between neighbourhoods of $x$ and $y$, and if a weight $\vph\in FW_y$, then the weight $\psi=F^\ast\vph\in FW_x$.
\end{example}

\section{Open questions}\label{sec:open}

Here we would like to mention a few questions that are, as we believe, quite important in understanding the nature of plurisubharmonic singularities.

\medskip
{\bf 1.}  {\sl Do there exist weights $\vph$ such that $d_\vph\equiv0$?}

\medskip
{\bf 2.} As pointed out in Example \ref{ex:12}, solutions to the extremal problem for homogeneous weights are exactly flat weight functions. {\sl Is $d_\vph$, for any  weight $\vph$, flat?}
It would be interesting to have an answer even in the case of weights with analytic singularities.

\medskip

{\bf 3.} Flatness means some regularity of the weight. What kind of regularity is it? Specifically, {\sl are flat weights tame? asymptotically  analytic? Do they have finite {\L}ojasiewicz exponent?} All known examples of flat weights have analytic or tame singularities.

\medskip
{\bf 4.} In Section~\ref{sec:tame} we presented a procedure for constructing the extremal functions for weights with analytic singularities. It rests on finding best plurisubharmonic minorants for certain tame weights, which makes the procedure somewhat implicit. {\sl Is it possible to describe explicitly the extremal functions for weights with analytic singularities?} One more question, {\sl do such extremal functions have analytic singularities? }

\medskip
{\bf 5.} The construction of the extremal function for analytic singularities $\phi$ is based on elementary representing weights $\phi_i$ for divisorial (or Rees) valuations. Such valuations play central role in investigation of singularities, as well as in many other problems of algebraic geometry and commutative algebra. {\sl Are elementary representing weights analytic? How are the asymptotics of $\phi_i$ related to the asymptotic of the weight $\phi$?}

\medskip
{\bf 6.} {\sl What are flat weights with analytic singularities? How can they be explicitly described?} So far, the only known flat analytic weights come from Examples \ref{ex:11} and \ref{ex:13}, modulo all holomorphic coordinate changes (as noticed in Example \ref{ex:14}).

\medskip
{\bf 7.} It was shown in Corollary \ref{cor:trop} that the condition $$\nu(\max_iu_i,\vph)=\min_i\nu(u_i,\vph)$$ for all $u_i$ implies flatness of $\vph$.
On the other hand, any flat weight $\vph$ satisfies
$$\sigma\left(\sum_iu_i,\vph\right)=\sum_i\sigma(u_i,\vph)$$
{\sl Is it a characteristic property for flat weights as well?}

\bigskip
{\it Acknowledgement.} The author is grateful to the Mathematics Department of Purdue University where part of the work was done.


\begin{thebibliography}{11}

\bibitem{BA}
{C. Bivi\`a-Ausina}, {Joint reductions of monomial ideals and multiplicity of complex analytic maps}, Math. Res. Lett. \textbf{15} (2008), no. 2, 389--407.

\bibitem{BFaJ}
{S. Boucksom, C. Favre and M. Jonsson}, {Valuations and
plurisubharmonic singularities}, Publ. Res. Inst. Math. Sci. \textbf{44} (2008), no. 2, 449--494.

\bibitem{D3}
{J.-P. Demailly}, {Mesures de Monge-Amp\`ere et
caract\'erisation g\'eom\'etrique des vari\'et\'es alg\'ebriques
affines}, M\'em. Soc. Math. France (N. S.) \textbf{19} (1985), 1-124.


\bibitem{D1}
{J.-P. Demailly}, {Nombres de Lelong g\'en\'eralis\'es,
th\'eor\`emes d'int\'egralit\'e et d'analycit\'e,} Acta Math. \textbf{159} (1987), 153--169.

\bibitem{D5}
{J.-P. Demailly}, {Mesures de Monge-Amp\`ere et mesures
plurisousharmoniques}, Math. Z. \textbf{194} (1987), 519--564.

\bibitem{D2}
{J.-P. Demailly}, {Regularization of closed positive currents
  and intersection theory}, J. Algebraic Geometry \textbf{1} (1992), 361--409.

\bibitem{D}
{J.-P. Demailly}, { Monge-Amp\`ere operators, Lelong numbers
and intersection theory,} Complex Analysis and Geometry (Univ.
Series in Math.), ed. by V. Ancona and A. Silva, Plenum Press, New
York 1993, 115--193.

\bibitem{D8}
{J.-P. Demailly}, {Estimates on Monge-Amp\`ere operators derived from a local algebra inequality}, preprint at http://arxiv.org/abs/0709.3524v2.


\bibitem{FaJ}
{C. Favre and M. Jonsson}, {Valuative analysis of planar
plurisubharmonic functions}, Invent. Math. \textbf{162} (2005),
271--311.

\bibitem{Kis1}
{C.O. Kiselman,} {Densit\'e des fonctions
plurisousharmoniques}, Bull. Soc. Math. France \textbf{107} (1979),
295--304.

\bibitem{Kis2}
{C.O. Kiselman}, {Un nombre de Lelong raffin\'e}, In:
S\'eminaire d'Analyse Complexe et G\'eom\'etrie 1985-87, Fac. Sci.
Monastir Tunisie 1987, 61--70.

\bibitem{LeR}
{P. Lelong and A. Rashkovskii}, {Local indicators for
plurisubharmonic functions}, J. Math. Pures Appl. \textbf{78} (1999),
233--247.


\bibitem{R1}
{A. Rashkovskii}, {Newton numbers and residual measures of
  plurisubharmonic functions}, Ann. Polon. Math. \textbf{75} (2000),
no. 3, 213-231.

\bibitem{R4}
{A. Rashkovskii}, {Lelong numbers with respect to regular
plurisubharmonic weights}, Results Math. \textbf{39} (2001), 320-332.

\bibitem{R7}
{A. Rashkovskii}, {Relative types and extremal problems for
plurisubharmonic functions}, Int. Math. Res. Not. \textbf{2006} (2006),
Article ID 76283, 26 p.


\bibitem{T}
{B. Teissier}, {Cycles \'evanescents, sections planes et conditions de Whitney}, Singularit\'es \`a Carg\`ese, Asterisque \textbf{7-8} (1973), 285--362.

\bibitem{Za0}
{V.P. Zahariuta}, {Spaces of analytic functions and maximal
plurisubharmonic functions.} D.Sci. Dissertation, Rostov-on-Don,
1984.

\bibitem{Za} {V.P. Zahariuta}, {Spaces of analytic functions and
Complex Potential Theory}, Linear Topological Spaces and Complex
Analysis \textbf{1} (1994), 74--146.



\end{thebibliography}
\end{document}